\documentclass{amsart}

\usepackage{amssymb,mathrsfs,enumitem}
\usepackage{bbm,mathtools}
\allowdisplaybreaks

\newcommand{\R}{\ensuremath{\mathbb{R}}}
\newcommand{\Exp}{\ensuremath{\mathbb{E}}}
\newcommand{\F}{\ensuremath{\mathcal{F}}}
\newcommand{\Sc}{\ensuremath{\mathcal{S}}}

\newcommand{\Lone}{\ensuremath{\mathcal{L}^1}}
\newcommand{\Ltwo}{\ensuremath{\mathcal{L}^2}}
\newcommand{\C}{\ensuremath{\mathcal{C}}}

\newcommand{\inpr}[3][]{\left\langle#2 \,,\, #3\right\rangle_{#1}}

\newcommand{\indicator}[1]{\mathbbm{1}_{#1}}

\theoremstyle{plain}
\newtheorem{thm}{Theorem}[section]

\newtheorem{lem}[thm]{Lemma}
\newtheorem{propn}[thm]{Proposition}
\theoremstyle{definition}
\newtheorem{rem}[thm]{Remark}
\newtheorem{defn}[thm]{Definition}
\newtheorem{eg}[thm]{Example}

\numberwithin{equation}{section}

\def\cprime{$'$} \def\cprime{$'$}
\providecommand{\bysame}{\leavevmode\hbox to3em{\hrulefill}\thinspace}
\providecommand{\MR}{\relax\ifhmode\unskip\space\fi MR }
\providecommand{\href}[2]{#2}

\begin{document}
\date{}
\title[Stationary solutions]{Stationary solutions of stochastic partial 
differential equations in the space of tempered distributions}
\author{Suprio Bhar}
\address{Suprio Bhar, Tata Institute of Fundamental Research - Centre for Applicable Mathematics, Bangalore, India.}
\email{suprio@tifrbng.res.in}

\begin{abstract}
In Rajeev (2013), `Translation invariant diffusion in the space of 
tempered distributions', it was shown that there is an one to one 
correspondence 
between solutions of a class of finite dimensional stochastic differential equations (SDEs) and solutions of a 
class of stochastic parial differential equations (SPDEs) in $\mathcal{S}'$, the space of tempered distributions, driven 
by the same Brownian motion. There the coefficients $\bar{\sigma}, \bar{b}$ of 
the finite dimensional SDEs were related to the coefficients of the SPDEs in 
$\mathcal{S}'$ in a special way, viz. through convolution with the initial 
value $y$ of the SPDEs.\\
In this paper, we consider the situation where the solutions of the 
finite dimensional SDEs are stationary and ask whether the corresponding 
solutions of the equations in $\mathcal{S}'$ are also stationary. We provide an 
affirmative answer, when the initial random variable takes value in a certain 
set $\mathcal{C}$, which ensures that the coefficients of the finite 
dimensional 
SDEs are related to the coefficients of the SPDEs in the above `special' manner.
\end{abstract}
\maketitle

\section{Introduction}
The topology of infinite dimensional spaces plays a major role in the study of 
stochastic differential equations in those spaces. These topological vector 
spaces are usually taken to be countably Hilbertian Nuclear spaces (see 
\cite{MR771478, MR1465436}) and in particular real separable Hilbert spaces 
(see \cite{MR2295103, MR1207136, MR2560625}). In \cite{MR3063763}, a 
correspondence was shown between finite dimensional stochastic differential 
equations and 
stochastic partial differential equations in $\Sc'(\R^d)$ (where $\Sc'(\R^d)$ 
denotes the space of tempered distributions) via an It\={o} formula. The 
results there involves deterministic 
initial conditions. In this paper we extend this correspondence to random 
initial conditions taking values in some Hermite Sobolev space $\Sc_p(\R^d)$. 
Assuming the 
existence of stationary solutions of finite dimensional stochastic differential 
equations, we then show the existence of stationary 
solutions of infinite dimensional stochastic partial differential equations, 
via an It\={o} formula which is used in proving the correspondence.

Let $(\Omega,\F,(\F_t),P)$ be a filtered complete probability space satisfying 
the usual 
conditions. 
Let $\{B_t\}$ be a $d$ dimensional $(\F_t)$ standard Brownian motion. By 
$(\F^{B}_t)$ we denote the filtration generated by $\{B_t\}$. Let $\delta$ be 
an arbitrary state, viewed as an isolated point of $\hat\Sc_p(\R^d):= 
\Sc_p(\R^d) \cup \{\delta\}$. For an initial 
condition $\psi \in 
\Sc_{p}(\R^d)$, consider the stochastic partial differential equation
\begin{equation}\label{spdeins'}
dY_t = A(Y_t).\,dB_t + L(Y_t)\, dt;\quad
Y_0 = \psi,
\end{equation}
where 
\begin{enumerate}[label=(\roman*)]
\item the operators $A:= (A_1,\cdots,A_d), L$ on $\Sc_{p}(\R^d)$ as follows: 
for $\phi \in \Sc_p(\R^d)$
\begin{equation}\label{op-A}
A_i \phi := - \sum_{j=1}^d \inpr[ji]{\sigma}{\phi}\, \partial_j \phi,\; 
i=1,\cdots,d
\end{equation}
and
\begin{equation}\label{op-L}
L\phi := \frac{1}{2} \sum_{i,j = 1}^d 
(\inpr{\sigma}{\phi}\inpr{\sigma}{\phi}^t)_{ij}\, \partial_{ij}^2 \phi - 
\sum_{i=1}^d \inpr[i]{b}{\phi}\, \partial_i \phi,
\end{equation}
\item $\sigma = (\sigma_{ij})_{d \times 
d}, b = (b_1,b_2, 
\cdots,b_d)$ with $\sigma_{ij},b_i \in \Sc_{-p}(\R^d)$ for $i,j = 1,2, \cdots, 
d$. For any $\phi \in \Sc_{p}(\R^d)$, by $\inpr{\sigma}{\phi}$ we denote 
the $d\times d$ matrix with entries $\inpr[ij]{\sigma}{\phi}:= 
\inpr{\sigma_{ij}}{\phi}$. Similarly $\inpr{b}{\phi}$ is a vector in $\R^d$ 
with $\inpr[i]{b}{\phi} := \inpr{b_i}{\phi}$.
\end{enumerate}
Given $\psi \in 
\Sc_p(\R^d)$, let $\bar\sigma(\cdot\,;\psi):\R^d\to \R^{d^2}$ 
and 
$\bar 
b(\cdot\,;\psi):\R^d\to\R^d$ be the following functions 
$\bar{\sigma}(x;\psi):=(\inpr{\sigma_{ij}}{\tau_x \psi})$ and 
$\bar{b}(x;\psi):=(\inpr{b_{i}}{\tau_x \psi})$. Let $\tau_x, x \in \R^d$ denote 
the translation operators (see Section 2). The next result is 
about the existence and 
uniqueness of a strong solution of \eqref{spdeins'}.

\begin{thm}[{\cite[Theorem 3.4 and Lemma 3.6]{MR3063763}}]
Let $\psi, \sigma_{ij},b_i,\{B_t\}$ be as above. Suppose that the functions $x 
\mapsto 
\bar{\sigma}(x;\psi)$ and $x \mapsto 
\bar{b}(x;\psi)$ are locally Lipschitz. Then equation 
\eqref{spdeins'} has a unique $\hat\Sc_p(\R^d)$ valued $(\F^{B}_t)$ adapted 
strong local solution given by
\[Y_t = \tau_{Z_t}(\psi), \, \text{ for }\, 0 \leq t < \eta,\]
where $\eta$ is an $(\F^B_t)$ adapted stopping time and $\{Z_t\}$ solves the 
stochastic differential equation
\begin{equation}\label{sdeinRd}
dZ_t = \bar{\sigma}(Z_t;\psi).\,dB_t + \bar{b}(Z_t;\psi)\, dt;\quad
Z_0 = 0.
\end{equation}
\end{thm}
In this paper, we use the same techniques as those used in 
\cite{MR3063763}. In Section 2, 
we list basic properties of Hermite 
Sobolev spaces which are used throughout the paper. In Section 3, we extend 
the It\={o} 
formula \cite[Theorem 
2.3]{MR1837298} to Theorem \ref{Ito-random} involving random initial condition 
and prove existence and uniqueness results for the 
solutions of finite dimensional stochastic differential equation 
\eqref{spdeinRd-randomcoeff} (see Theorem 
\ref{suff-ext-unq-z2} and Theorem \ref{suff-loc-lip}). These results allow us 
to extend the said correspondence (see Theorem \ref{exstunq-spdeins'-init}, 
Lemma \ref{soln-characterization}, Theorem \ref{exstunq-spdeins'-init-loclip}). 
Note that the equation for $Z$ involves the initial condition for $Y$ i.e. 
$Y_0$, but with $Z_0 = 0$. To discuss existence and uniqueness for $Z$, we use 
`Lipschitz' criteria, which depends on $Y_0$. We need control on the norm of 
$Y_0$ to make the usual proof via Picard iteration work. In Proposition 
\ref{Yt-cont-xi} and Proposition \ref{Yt-cont-xi2} we prove $\Ltwo$ estimates 
on the supremum of the norms of the solutions of \eqref{spdein-sprime-init}, in 
terms of the initial condition.\\
In Section 4, we first show that $\R^d$ valued stationary processes can be lifted to $\Sc_p(\R^d)$ valued stationary processes via the translations operators $\tau_x$ (see Proposition \ref{lifting-stationary}) and using which we then prove the existence of a stationary solution of infinite dimensional stochastic partial differential equation \eqref{spdeins'-splinit} given that the corresponding finite dimensional stochastic differential equation \eqref{spdeinRd-randomcoeff} has the same (see Theorem \ref{stationary-basic-existence}).  We present a method to lift stationary solutions of (possibly) unrelated finite dimensional stochastic differential equations to stationary solutions of \eqref{spdeins'-splinit}. We do this by describing conditions on a random variable $\xi$ that appears in the initial condition of \eqref{spdeins'-splinit}. We define a subset $\C$ of the Hermite Sobolev space with the 
following property: if the random variable $\xi$ takes values in the set 
$\C$, then the corresponding finite dimensional stochastic differential equations 
are all the same. This property is observed in Lemma \ref{barsigmab-eq-sigmab} 
and using which we construct stationary solutions of stochastic 
partial differential equations in our class (Theorem 
\ref{stationary-existence}). To guarantee non-explosion for finite dimensional 
stochastic differential equations with locally Lipschitz coefficients, we use a 
`Liapunov' type criteria (\cite[7.3.14 Corollary]{MR1267569}). Two examples 
of stationary solutions are given in Example \ref{eg-st-soln} and in 
Proposition \ref{stationary-estimate}, we 
obtain $\Lone$ estimates 
on the supremum of the norms of the stationary solutions, in 
terms of the initial condition.
\section{Topologies on $\Sc$ and $\Sc'$}
Let $\Sc(\R^d)$ be the space of smooth rapidly 
decreasing
$\R$-valued functions on $\R^d$ with the topology given by L. Schwartz (see
\cite{MR2296978}) and let $\Sc'(\R^d)$ be the dual space, known as the 
space of tempered distributions. For any $p \in \R$, let $\Sc_p(\R^d)$ be the 
completion of 
$\Sc(\R^d)$ in the inner product $\inpr[p]{\cdot}{\cdot}$ which is defined in 
terms of the $\Ltwo(\R^d)$ inner product (see
\cite[Chapter 1.3]{MR771478} for the details). The spaces $\Sc_p(\R^d), p 
\in 
\R$ are separable Hilbert spaces and are known as the Hermite-Sobolev spaces. 
We write $\Sc, \Sc', \Sc_p$ instead of $\Sc(\R), \Sc'(\R), \Sc_p(\R)$.\\
Note that $\Sc_0(\R^d) = \Ltwo(\R^d)$ and for 
$p>0$, $\Sc_p(\R^d) \subset \Ltwo(\R^d)$ (i.e. these distributions are given 
by functions) and $(\Sc_{-p}(\R^d), \|\cdot\|_{-p})$ is dual to $(\Sc_p(\R^d),
\|\cdot\|_p)$. Furthermore,
\[\Sc(\R^d) = \bigcap_{p \in \R}(\Sc_p(\R^d), \|\cdot\|_p), \quad 
\Sc'(\R^d) = \bigcup_{p \in \R}(\Sc_p(\R^d), \|\cdot\|_p)
\]
Given $\psi \in \Sc(\R^d)$ (or $\Sc_p(\R^d)$) and $\phi \in \Sc'(\R^d)$
(or
$\Sc_{-p}(\R^d)$), the action of $\phi$ on $\psi$ will be denoted by
$\inpr{\phi}{\psi}$.\\
Let $\{h_n: n \in {\mathbb Z}_+^d\}$ be the Hermite functions (see \cite[Chapter 
1.3]{MR771478}), where ${\mathbb 
Z}_+^d := \{n= (n_1,\cdots, n_d) : n_i\, 
\text{non-negative integers}\}$.
If $n = (n_1, \cdots ,n_d)$, we define $|n|
:= n_1+\cdots +n_d$. Note that $\{h_n^p: n \in {\mathbb Z}_+^d\}$ 
forms an orthonormal basis for $\Sc_p(\R^d)$, where $h_n^p:= (2|n|+d)^{-p} 
h_n$.\\
Consider the derivative maps denoted by $\partial_i:\Sc(\R^d)\to
\Sc(\R^d)$ for $i=1,\cdots,d$. We can extend these maps by duality to
$\partial_i:\Sc'(\R^d) \to \Sc'(\R^d)$ as follows: for $\psi \in
\Sc'(\R^d)$,
\[\inpr{\partial_i \psi}{\phi}:=-\inpr{\psi}{\partial_i \phi}, \; \forall \phi
\in \Sc(\R^d).\]
Let $\{e_i: i=1,\cdots,d\}$ be the standard basis vectors in $\R^d$. Then 
for any $n =
(n_1,\cdots,n_d) \in {\mathbb Z}_+^d$ we have (see
\cite[Appendix A.5]{MR562914})
\[\partial_i h_n =
\sqrt{\frac{n_i}{2}}h_{n-e_i}-\sqrt{\frac{n_i+1}{2}}h_{n+e_i},\]
with the convention that for a multi-index $n = (n_1,\cdots,n_d)$, if $n_i
< 0$ for some $i$, then $h_n \equiv 0$. Above recurrence implies that
$\partial_i:\Sc_{p}(\R^d)\to\Sc_{p-\frac{1}{2}}(\R^d)$ is a bounded linear
operator.\\
For $x \in \R^d$, let $\tau_x$ denote the translation operators on $\Sc(\R^d)$ 
defined by
$(\tau_x\phi)(y):=\phi(y-x), \, \forall y \in \R^d$. This operators can be
extended to $\tau_x:\Sc'(\R^d)\to \Sc'(\R^d)$ by
\[\inpr{\tau_x\phi}{\psi}:=\inpr{\phi}{\tau_{-x}\psi},\, \forall \psi \in
\Sc(\R^d).\]
\begin{lem}\label{tau-x-estmte}
The translation operators $\tau_x, x \in \R^d$ have the following properties:
\begin{enumerate}[label=(\alph*),ref=\ref{tau-x-estmte}(\alph*)]
\item\label{tau-x-bnd} (\cite[Theorem 2.1]{MR1999259}) For $x \in \R^d$ and any 
$p \in \R$, $\tau_x: 
\Sc_p(\R^d)\to\Sc_p(\R^d)$
is a bounded linear map. In particular, there exists a real polynomial $P_k$ of
degree $k = 2([|p|]+1)$ such that
\[\|\tau_x\phi\|_p\leq P_k(|x|)\|\phi\|_p, \, \forall \phi \in \Sc_p(\R^d),\]
where $|x|$ denotes the standard Euclidean norm of $x$.
\item\label{commut-tau-partial} For $x \in \R^d$ and any $i=1,\cdots,d$ we 
have $\tau_x\partial_i = \partial_i\tau_x$.
\item Fix $\phi \in \Sc_p(\R^d)$. Then $x \mapsto \tau_x\phi$ is continuous.
\end{enumerate}
\end{lem}
\begin{proof}
Part $(b)$ can be verified for elements of $\Sc(\R^d)$ first and then can be 
extended to elements of $\Sc'(\R^d)$ via duality.\\
Proof of part $(c)$ is contained in in the proof of \cite[Proposition 
3.1]{MR2373102}.
\end{proof}
On $\Sc(\R^d)$ consider the multiplication operators $M_i, i=1,\cdots,d$ 
defined by
\[(M_i\phi)(x):= x_i\phi(x),\, \phi \in \Sc(\R^d), x=(x_1,\cdots,x_d) \in 
\R^d.\]
By duality these operators can be extended to $M_i:\Sc'(\R^d)\to\Sc'(\R^d)$. 
Note that 
$x_i h_n(x) = \sqrt{\frac{n_i+1}{2}}h_{n+e_i}(x) 
+\sqrt{\frac{n_i}{2}}h_{n-e_i}(x)$ (see \cite[Appendix A.5, 
equation (A.26)]{MR562914}) and hence 
$M_i:\Sc_p(\R^d) \to \Sc_{p-\frac{1}{2}}(\R^d)$ is a bounded linear operator, 
for any $p \in 
\R$. For dimension $d=1$, we write $M_x$ instead of $M_1$.

\section{Stochastic Differential Equations in $\Sc'$}
We use the following terminology and notation. We say $\{\sigma_n\}$ is a 
localizing sequence, if each 
$\sigma_n$ is an $(\F_t)$ stopping time with $\sigma_n\uparrow \infty$. We use 
stochastic integration in the Hilbert spaces $\Sc_p(\R^d)$ (as in 
\cite{MR3063763}). $\Sc_p(\R^d)$ valued stochastic 
integrals $\int_0^t G_s\, dX_s$ can be defined for $\Sc_p(\R^d)$ valued 
predictable, locally norm-bounded processes $\{G_t\}$ and real semimartingales 
$\{X_t\}$. For any $x \in \R^n$, 
by $|x|$ we denote the standard
Euclidean norm of $x$. The dimension will be clear from the context where this 
notation is used. 

\subsection{Random Initial Conditions}
Let $\xi$ be an $\Sc_p(\R^d)$ valued $\F_0$-measurable random variable. Now 
consider the stochastic partial 
differential equation
\begin{equation}\label{spdein-sprime-init}
dY_t = A(Y_t).\,dB_t + L(Y_t)\, dt; \quad
Y_0 = \xi,
\end{equation}
where the operators $A=(A_1,\cdots,A_d), L$ are as in \eqref{op-A}, \eqref{op-L}. We want to extend the results of \cite[Section 3]{MR3063763} to the case of 
$\Sc_p(\R^d)$ valued random 
initial conditions. We have the following It\={o} formula (see \cite[Theorem 
2.3]{MR1837298}).
\begin{propn}\label{Ito-translates}
Let $p \in \R$ and $\phi \in \Sc_{-p}(\R^d)$. Let $X=(X^1,\cdots,X^d)$ be an 
$\R^d$
valued continuous $(\F_t)$ adapted semimartingale. Then we have the following 
equality in 
$\Sc_{-p-1}(\R^d)$, a.s.
\begin{equation}
\tau_{X_t}\phi = \tau_{X_0}\phi - \sum_{i=1}^d \int_0^t
\partial_i\tau_{X_{s}}\phi\,
dX^i_s + \frac{1}{2}\sum_{i,j=1}^d \int_0^t \partial_{ij}^2\tau_{X_{s}}\phi\,
d[X^i,X^j]_s,\, \forall t \geq 0.
\end{equation}
\end{propn}
We need to extend above result to allow random $\phi$.
\begin{thm}\label{Ito-random}
Let $p\ \in \R$. Let $\xi$ be an $\Sc_p(\R^d)$ valued $\F_0$-measurable random 
variable with $\Exp \|\xi\|_{p}^2 < \infty$. Let $X=(X^1,\cdots,X^d)$ be an 
$\R^d$ 
valued continuous 
semimartingale. Then we have the following equality in 
$\Sc_{p-1}(\R^d)$, a.s.
\begin{equation}
\tau_{X_t}\xi = \tau_{X_0}\xi - \sum_{i=1}^d \int_0^t
\partial_i\tau_{X_{s}}\xi\,
dX^i_s + \frac{1}{2}\sum_{i,j=1}^d \int_0^t \partial_{ij}^2\tau_{X_{s}}\xi\,
d[X^i,X^j]_s, \, \forall t \geq 0.
\end{equation}
\end{thm}
\begin{proof}
Fix $\phi \in \Sc(\R^d)$. Then $\phi \in \Sc_{-p+1}(\R^d)$ and by the previous 
Proposition, we have in $\Sc_{-p}(\R^d)$ a.s. for all $ t \geq 0$
\[\tau_{-X_t}\phi = \tau_{-X_0}\phi + \sum_{i=1}^d \int_0^t
\partial_i\tau_{-X_{s}}\phi\,
dX^i_s + \frac{1}{2}\sum_{i,j=1}^d \int_0^t \partial_{ij}^2\tau_{-X_{s}}\phi\,
d[X^i,X^j]_s.\]
Then a.s.
\begin{equation}\label{act-xi-Xt-phi}
\begin{split}
\inpr{\xi}{\tau_{-X_t}\phi}
= &\inpr{\xi}{\tau_{-X_0}\phi} + \inpr{\xi}{\sum_{i=1}^d \int_0^t
\partial_i\tau_{-X_{s}}\phi\,
dX^i_s}\\
&+\inpr{\xi}{\frac{1}{2}\sum_{i,j=1}^d \int_0^t 
\partial_{ij}^2\tau_{-X_{s}}\phi\,
d[X^i,X^j]_s},\, \forall t \geq 0.
\end{split}
\end{equation}
Now using \cite[Proposition 1.3(a)]{MR1837298} and Lemma 
\ref{commut-tau-partial}, we can show
\[\inpr{\xi}{\sum_{i=1}^d \int_0^t
\partial_i\tau_{-X_{s}}\phi\,
dX^i_s} 
=\inpr{-\sum_{i=1}^d \int_0^t
\partial_i\tau_{X_{s}}\xi\,
dX^i_s}{\phi}\]
and
\[\inpr{\xi}{\frac{1}{2}\sum_{i,j=1}^d \int_0^t 
\partial_{ij}^2\tau_{-X_{s}}\phi\,
d[X^i,X^j]_s}=\inpr{\frac{1}{2}\sum_{i,j=1}^d \int_0^t 
\partial_{ij}^2\tau_{X_{s}}\xi\,
d[X^i,X^j]_s}{\phi}.\] 
Using \eqref{act-xi-Xt-phi} we get a $P$-null set $\mathcal{N}$ such that for 
$\omega \in 
\Omega\setminus\mathcal{N}$ and for any multi-index $n=(n_1,\cdots,n_d)$ we have
\begin{align*}
\inpr{\tau_{X_t}\xi}{h_n}=&\inpr{\tau_{X_0}\xi}{h_n}-\inpr{\sum_{i=1}^d 
\int_0^t
\partial_i\tau_{X_{s}}\xi\,
dX^i_s}{h_n}\\
&+\inpr{\frac{1}{2}\sum_{i,j=1}^d \int_0^t 
\partial_{ij}^2\tau_{X_{s}}\xi\,
d[X^i,X^j]_s}{h_n},\,\forall t \geq 0
\end{align*}
where $h_n$ are the Hermite functions which form a total set in 
$\Sc_{p-1}(\R^d)$. Since 
$\{\tau_{X_t}\xi -\tau_{X_0}\xi + \sum_{i=1}^d 
\int_0^t
\partial_i\tau_{X_{s}}\xi\,
dX^i_s
-\frac{1}{2}\sum_{i,j=1}^d \int_0^t 
\partial_{ij}^2\tau_{X_{s}}\xi\,
d[X^i,X^j]_s\}$ is an $\Sc_{p-1}(\R^d)$ valued process, we have the equality 
in $\Sc_{p-1}(\R^d)$ a.s.
\[\tau_{X_t}\xi -\tau_{X_0}\xi + \sum_{i=1}^d 
\int_0^t
\partial_i\tau_{X_{s}}\xi\,
dX^i_s
-\frac{1}{2}\sum_{i,j=1}^d \int_0^t 
\partial_{ij}^2\tau_{X_{s}}\xi\,
d[X^i,X^j]_s = 0,\, t \geq 0.\]
This completes the proof.
\end{proof}
\begin{proof}[Alternative proof of Theorem \ref{Ito-random}]
We make two observations, including a 
property of stochastic 
integrals, viz. \eqref{F0sets-in-st-intg}.
\begin{enumerate}[label=(\alph*)]
\item Given any $\F_0$-measurable set $F$, an $\Sc_p(\R^d)$ valued predictable 
step process $\{G_t\}$ and a continuous $\R^d$ valued semimartingale $\{X_t\}$, 
we have a.s.
\begin{equation}\label{F0sets-in-st-intg}
\indicator{F}\int_0^t G_s\, dX_s = \int_0^t \indicator{F}G_s\, dX_s,\, t 
\geq 0.
\end{equation}
Above equality can be extended to the case involving norm-bounded $\Sc_p(\R^d)$ 
valued predictable process $\{G_t\}$.
\item Given any $\F_0$-measurable set $F$, $\phi \in \Sc_p(\R^d)$, $\psi \in 
\Sc(\R^d)$ and $x \in \R^d$ we have
\begin{equation}\label{translation-indicator}
\begin{split}
\inpr{\indicator{F}\tau_{x}\phi}{\psi} &= 
\indicator{F}\inpr{\tau_{x}\phi}{\psi} = 
\indicator{F}\inpr{\phi}{\tau_{-x}\psi}\\
&=\inpr{\indicator{F}\phi}{\tau_{-x}\psi}=\inpr{\tau_{x}(\indicator{F}\phi)
}{\psi}
\end{split}
\end{equation}
and hence $\indicator{F}\tau_{x}\phi = \tau_{x}(\indicator{F}\phi)$. 
Similarly $\indicator{F}\tau_{x}\phi = 
\tau_{\indicator{F}x}(\indicator{F}\phi)$.
\end{enumerate}
Using Proposition \ref{Ito-translates} and equations \eqref{F0sets-in-st-intg}, 
\eqref{translation-indicator}, we can establish Theorem \ref{Ito-random} when 
$X$ is bounded and $\xi$ is an $\Sc_p(\R^d)$ valued $\F_0$-measurable simple function. When $\xi$ is square integrable, the result can be 
proved by approximating $\xi$ using $\Sc_p(\R^d)$ valued $\F_0$-measurable simple functions. Finally, via localization under stopping times, we 
can prove the result for unbounded $X$.
\end{proof}

We need an existence and uniqueness of solution to the following equation:
\begin{equation}\label{spdeinRd-randomcoeff}
dZ_t = \bar{\sigma}(Z_t;\xi).\,dB_t + \bar{b}(Z_t;\xi)\, dt;\quad Z_0 = \zeta,
\end{equation}
where $\xi$ is an $\Sc_p(\R^d)$ valued $\F_0$-measurable random 
variable and $\zeta$ is an $\R^d$ valued $\F_0$-measurable random 
variable. Unless stated otherwise, we assume that both 
$\xi,\zeta$ 
are independent of the Brownian motion $\{B_t\}$. Let $(\mathcal{G}_t)$ denote 
the filtration generated by $\xi,\zeta$ and $\{B_t\}$. Let 
$\mathcal{G}_{\infty}$ denote the smallest sub $\sigma$-field of $\F$ 
containing $\mathcal{G}_t$ for all $t \geq 0$. Let $\mathcal{G}_{\infty}^P$ 
be the $P$-completion of $\mathcal{G}_{\infty}$ and let $\mathcal{N}^P$ be the 
collection of all $P$-null sets in $\mathcal{G}_{\infty}^P$. Define
\[\F^{\xi,\zeta}_t:=\bigcap_{s > t} \sigma(\mathcal{G}_s \cup 
\mathcal{N}^P),\, t \geq 0\]
where $\sigma(\mathcal{G}_s \cup 
\mathcal{N}^P)$ denotes the smallest $\sigma$-field generated by the 
collection $\mathcal{G}_s \cup 
\mathcal{N}^P$. This filtration satisfies the usual conditions.
$\F^{\xi,\zeta}_\infty$ will denote the $\sigma$ field generated by the 
collection $\bigcup_{t \geq 0} \F^{\xi,\zeta}_t$. If 
$\zeta$ is a constant, then we write $(\F^{\xi}_t)$ instead of 
$(\F^{\xi,\zeta}_t)$.
\begin{propn}\label{suff-ext-unq-z}
Suppose the following conditions are satisfied.
\begin{enumerate}
[label=(\roman*)]
\item $\xi$ is norm-bounded in $\Sc_p(\R^d)$, i.e. there exists a constant $K > 
0$ such that $\|\xi\|_p \leq K$.
\item $\Exp |\zeta|^2 < \infty$.
\item (Globally Lipschitz in x, locally in y) For any fixed $y \in 
\Sc_p(\R^d)$, the functions $x 
\mapsto 
\bar{\sigma}(x;y)$ and $x 
\mapsto 
\bar{b}(x;y)$ are globally Lipschitz functions in $x$ and the Lipschitz 
coefficient is independent of $y$ when $y$ varies over any bounded set $G$ in 
$\Sc_p(\R^d)$; i.e. for any bounded set $G$ in 
$\Sc_p(\R^d)$ there exists a constant $C(G) > 0$ such 
that for all $x_1,x_2\in \R^d, y \in G$
\[|\bar{\sigma}(x_1;y) - \bar{\sigma}(x_2;y)|+|\bar{b}(x_1;y) - \bar{b}(x_2;y)| 
\leq C(G) |x_1-x_2|.\]
\end{enumerate}
Then equation \eqref{spdeinRd-randomcoeff} has a continuous 
$(\F^{\xi,\zeta}_t)$ adapted 
strong solution $\{X_t\}$ with the 
property 
that $\Exp \int_0^T |X_t|^2\,dt < \infty$ for any $T>0$. Pathwise uniqueness of 
solutions also holds, i.e. if $\{X^1_t\}$ is another solution, then $P(X_t = 
X^1_t, \, t \geq 0)=1$.
\end{propn}
\begin{proof}
We follow the proof in \cite[Theorem 5.2.1]{MR2001996} with appropriate 
modifications. First we show the uniqueness of the strong solution.\\
Let $\{Z^1_t\}$ and $\{Z^2_t\}$ be two strong solutions of 
\eqref{spdeinRd-randomcoeff}. Define two processes $a(t,\omega) := 
\bar\sigma(Z^1_t(\omega);\xi(\omega))-\bar\sigma(Z^2_t(\omega);\xi(\omega))$ 
and $\gamma(t,\omega) := 
\bar b(Z^1_t(\omega);\xi(\omega))-\bar b(Z^2_t(\omega);\xi(\omega))$. Since 
$\xi$ is norm-bounded, then by our 
hypothesis
\[|a(t,\omega)|^2 \leq C^2 \left|Z^1_t(\omega)-Z^2_t(\omega)\right|^2,\, 
|\gamma(t,\omega)|^2 \leq C^2 \left|Z^1_t(\omega)-Z^2_t(\omega)\right|^2\]
with $C = C(Range(\xi))$. Using It\={o} isometry and Cauchy-Schwarz Inequality, 
we get
\begin{equation}
\begin{split}
\Exp \left|Z^1_t-Z^2_t\right|^2 &\leq 2\,\Exp\int_0^t 
|a(s)|^2\,ds + 2t\,\Exp\int_0^t |\gamma(s)|^2\,ds\\
&\leq 2C^2(1+t) \int_0^t\Exp\left|Z^1_s-Z^2_s\right|^2 \,ds
\end{split}
\end{equation}
By Gronwall's inequality, $\Exp \left|Z^1_t-Z^2_t\right|^2 = 0, \forall 
t \geq 0$. This proves the uniqueness.\\
For the existence of solution we use a Picard type 
iteration. Set $Z^{(0)}_t = \zeta$ and then 
successively define
\begin{equation}\label{iteration-xi}
Z^{(k+1)}_t := \zeta+\int_0^t \bar\sigma(Z^{(k)}_s;\xi)\,dB_s + \int_0^t \bar 
b(Z^{(k)}_s;\xi)\,ds,\, \forall k \geq 0.
\end{equation}
Fix any compact time interval $[0,N]$. For $k \geq 1, t \in [0,N]$ we have
\begin{equation}\label{iteration-esmt}
\Exp|Z^{(k+1)}_t-Z^{(k)}_t|^2 \leq 2C^2(1+N)\int_0^t 
\Exp|Z^{(k)}_s-Z^{(k-1)}_s|^2\, ds.
\end{equation}
Using the Lipschitz continuity for any $x \in \R^d, y \in 
Range(\xi)$ we have, $|\bar{\sigma}(x;y) - \bar{\sigma}(0;y)|+|\bar{b}(x;y) - 
\bar{b}(0;y)| 
\leq C |x|$. But $|\bar{\sigma}(0;y)| = |\!\inpr{\sigma}{y}\!| \leq 
\|\sigma_{ij}\|_{-p}\|y\|_{p}$ and $|\bar 
b(0;y)|=|\!\inpr{b}{y}\!| 
\leq \|b_i\|_{-p}\|y\|_p$. This shows $\bar\sigma, \bar b$ has 
linear growth in $x$, i.e. there exists a constant $D=D(Range(\xi))>0$ such that 
$|\bar{\sigma}(x;y)| \leq D(1+|x|), |\bar b(x;y)| \leq D(1+|x|)$ for $x \in 
\R^d,y \in Range(\xi)$. Since $Z^{(0)}_t = \zeta$ using \eqref{iteration-esmt} 
we get
\begin{equation}\label{base-step-esmt}
\begin{split}
\Exp|Z^{(1)}_t-Z^{(0)}_t|^2 &\leq 2\,\Exp\int_0^t 
|\bar\sigma(\zeta;\xi)|^2\,ds + 2t\, \Exp\int_0^t |\bar 
b(\zeta;\xi)|^2\,ds\\
&\leq 4D^2(1+N)(1+\Exp |\zeta|^2)t,\,\forall t \in [0,N].
\end{split}
\end{equation}
Now we use an induction on $k$ 
with \eqref{iteration-esmt} as the recurrence relations and 
\eqref{base-step-esmt} as our base step. Then there exists a constant $R>0$ 
such 
that
\begin{equation}\label{l2-p-increments}
\Exp|Z^{(k+1)}_t-Z^{(k)}_t|^2 \leq \frac{(Rt)^{k+1}}{(k+1)!}, \, \forall k 
\geq 0, t\in [0,N].
\end{equation}
Let $\lambda$ denote the Lebesgue 
measure on $[0,N]$.
We are going to show that the iteration converges in $\Ltwo(\lambda\times 
P)$ and the limit satisfy \eqref{spdeinRd-randomcoeff}. As in \cite[Theorem 
5.2.1]{MR2001996} we can show \[\|Z^{(m)}-Z^{(n)}\|_{\Ltwo(\lambda\times 
P)}  = 
\sum_{k=n}^{m-1} 
\left(\frac{(RN)^{k+2}}{(k+2)!}\right)^{\frac{1}{2}}\]
where $m,n$ are positive integers with $m > n$. 
Observe that $\|Z^{(m)}-Z^{(n)}\|_{\Ltwo(\lambda\times 
P)} \to 0$ as $m,n \to \infty$. Using completeness of 
$\Ltwo(\lambda\times 
P)$ we have a limit, which we denote by $\{X_t\}_{t \in [0,N]}$. Using 
\eqref{l2-p-increments}, we also have $\lim_{n \to 
\infty} Z^{(n)}_t \overset{\Ltwo(P)}{=} X_t$ for each $t \in [0,N]$.\\
This $\{X_t\}$ 
is measurable and $(\F^{\xi,\zeta}_t)$ adapted. Now using the linear growth of 
$x \mapsto \bar\sigma(x;y)$ (for every fixed $y \in \Sc_p(R^d)$) we have
\begin{align*}
\Exp\int_0^N \bar\sigma(X_s;\xi)^2 \, ds &\leq D^2\, \Exp\int_0^N 
(1+|X_s|)^2 \, ds\\
&\leq 2D^2\, \Exp\int_0^N 
(1+|X_s|^2) \, ds = 2D^2N + 2D^2 \|X\|^2_{\Ltwo(\lambda\times 
P)} < \infty.
\end{align*}
For Brownian motion, we can define stochastic integrals for adapted integrands 
satisfying the above integrability condition (see \cite[Chapter 3, Remark 
2.11]{MR1121940}). Hence $\{\int_0^t 
\bar\sigma(X_s;\xi)\, dB_s\}_{t \in [0,N]}$ exists. Since $\Exp \int_0^N 
|X_s|^2\, ds < \infty$, we 
have $\int_0^N |X_s|^2\, ds < \infty$ almost surely. Using the linear 
growth of 
$x \mapsto \bar b(x;y)$ (for every fixed $y \in \Sc_p(R^d)$) and 
Cauchy-Schwarz inequality, the existence of the process 
$\{\int_0^t 
\bar b(X_s;\xi)\, ds\}_{t \in [0,N]}$ can be established. Using It\={o} 
isometry and Lipschitz continuity of $\bar{\sigma}$ we get
\[\Exp\left|\int_0^t\bar\sigma(Z^{(k)}_s;\xi)\,dB_s -  \int_0^t 
\bar\sigma(X_s;\xi)\, dB_s\right|^2 \leq C^2\,\Exp \int_0^N |Z^{(k)}_s-X_s|^2 \, 
ds.\]
Using Jensen's inequality and the Lipschitz property of $\bar{b}$ we get
\[\Exp\left|\int_0^t\bar b(Z^{(k)}_s;\xi)\,ds- \int_0^t 
\bar b(X_s;\xi)\, ds\right|^2  \leq C^2N\,\Exp \int_0^N |Z^{(k)}_s-X_s|^2 \, 
ds.\]
Using above estimates, for each $t \in [0,N]$ we have
\[\int_0^t\bar\sigma(Z^{(k)}_s;\xi)\,dB_s 
\xrightarrow[k\to\infty]{\Ltwo(P)}\int_0^t\bar\sigma(X_s;\xi)\,dB_s,\]
and
\[\int_0^t\bar b(Z^{(k)}_s;\xi)\,dB_s 
\xrightarrow[k\to\infty]{\Ltwo(P)}\int_0^t\bar b(X_s;\xi)\,dB_s.\]
From \eqref{iteration-xi} we conclude that for each $t \in [0,N]$, a.s.
\[
X_t = \zeta+\int_0^t \bar\sigma(X_s;\xi)\,dB_s + \int_0^t \bar 
b(X_s;\xi)\,ds.
\]
The integral $\int_0^t \bar\sigma(X_s;\xi)\,dB_s$ has a continuous version 
(see \cite[Theorem 3.2.5]{MR2001996}). We denote the 
continuous version of 
$\{\zeta+\int_0^t 
\bar\sigma(X_s;\xi)\,dB_s + \int_0^t \bar 
b(X_s;\xi)\,ds\}_{t\in [0,N]}$ by $\{\widetilde X_t\}_{t\in [0,N]}$. Then for 
each $t \in [0,N]$, a.s.
\[\widetilde X_t = \zeta+\int_0^t \bar\sigma(X_s;\xi)\,dB_s + \int_0^t \bar 
b(X_s;\xi)\,ds = X_t,\, \text{a.s.}\]
In particular, for all $t \in [0,N]$ we have $\Exp |X_t - \widetilde X_t|^2 = 
0$. Then using the Lipschitz property of $\bar{\sigma}$, $\int_0^t 
\bar\sigma(X_s;\xi)\,dB_s = \int_0^t 
\bar\sigma(\widetilde X_s;\xi)\,dB_s$ a.s. We can also show $\int_0^t 
\bar b(X_s;\xi)\,dB_s = \int_0^t 
\bar b(\widetilde X_s;\xi)\,dB_s$ a.s. for each $t \in [0,N]$. Then for 
each $t \in [0,N]$, a.s.
\begin{align*}
\widetilde X_t &= \zeta+\int_0^t \bar\sigma(\widetilde X_s;\xi)\,dB_s + 
\int_0^t \bar 
b(\widetilde X_s;\xi)\,ds,\, \text{a.s.}
\end{align*}
Since $\{\widetilde X_t\}$ is continuous, we have, a.s.
\[\widetilde X_t =  \zeta+\int_0^t \bar\sigma(\widetilde X_s;\xi)\,dB_s + 
\int_0^t 
\bar 
b(\widetilde X_s;\xi)\,ds,\, t \in [0,N].\]
So we have obtained a continuous $(\F^{\xi,\zeta}_t)$ adapted solution for 
time interval $[0,N]$ for any positive integer $N$. The uniqueness of this 
continuous solution 
follows from the proof of uniqueness given at the beginning of this proof.\\
Let 
$\{X^{(N)}_t\}$ and $\{X^{(N+1)}_t\}$ be the solutions up to time $N$ and $N+1$ 
respectively. Then $\{X^{(N+1)}_{t \in [0,N]}\}$ is also a continuous solution 
up to time $N$ and hence by the uniqueness, is indistinguishable from 
$\{X^{(N)}_t\}$ on $[0,N]$. Using this consistency, we can patch up the 
solutions $\{X^{(N)}_t\}$ to obtain the solution of 
\eqref{spdeinRd-randomcoeff} on the time interval 
$[0,\infty)$.
\end{proof}
We now come to a main result regarding the existence and uniqueness of 
solutions of \eqref{spdeinRd-randomcoeff}. \begin{thm}\label{suff-ext-unq-z2}
Suppose the following are satisfied.
\begin{enumerate}
[label=(\roman*)]
\item $\Exp\|\xi\|_p^2 < \infty$.
\item $\zeta=c$, where $c$ is some element in $\R^d$.
\item (Globally Lipschitz in x, locally in y) For any fixed $y \in 
\Sc_p(\R^d)$, the functions $x 
\mapsto 
\bar{\sigma}(x;y)$ and $x 
\mapsto 
\bar{b}(x;y)$ are globally Lipschitz functions in $x$ and the Lipschitz 
coefficient is independent of $y$ when $y$ varies over any bounded set $G$ in 
$\Sc_p(\R^d)$; i.e. for any bounded set $G$ in 
$\Sc_p(\R^d)$ there exists a constant $C(G) > 0$ such 
that for all $x_1,x_2\in \R^d, y \in G$
\[|\bar{\sigma}(x_1;y) - \bar{\sigma}(x_2;y)|+|\bar{b}(x_1;y) - \bar{b}(x_2;y)| 
\leq C(G) |x_1-x_2|.\]
\end{enumerate}
Then equation \eqref{spdeinRd-randomcoeff} has a continuous $(\F^{\xi}_t)$ 
adapted 
strong solution $\{X_t\}$ such that 
there exists a localizing 
sequence of stopping times $\{\eta_n\}$ with $\Exp \int_0^{T\wedge\eta_n} 
|X_t|^2\,dt < \infty$ for any $T>0$. Pathwise uniqueness of 
solutions also holds, i.e. if $\{\tilde X_t\}$ is another solution, then $P(X_t 
= 
\tilde X_t, \, t \geq 0)=1$.
\end{thm}

\begin{rem}
Theorem \ref{suff-ext-unq-z2} is also true if $\zeta$ is an $\R^d$ 
valued 
$\F_0$-measurable square integrable random variable, which is also independent 
of the Brownian motion $\{B_t\}$. However, we only need the version for $\zeta 
= 0$ (see Theorem \ref{exstunq-spdeins'-init}).
\end{rem}

\begin{proof}[Proof of Theorem \ref{suff-ext-unq-z2}]
For all positive integers $k$, define $\xi^{(k)}:=\xi \indicator{(\|\xi\|_p 
\leq k)}$. Note that $\xi^{(k)}\xrightarrow[\Ltwo]{k\to\infty}\xi$ and the 
convergence is 
also 
almost sure. Also, $\indicator{(\|\xi\|_p\leq 
k)}\xi^{(k+1)}=\xi^{(k)}$. By \eqref{translation-indicator}, we have for any $x 
\in \R^d, y \in 
\Sc_p(\R^d), F \in \F$
\begin{equation}\label{bar-sigma-b-indicator}
\begin{split}
&\indicator{F}\bar \sigma(x;y) = \bar \sigma(x;\indicator{F}y) = \bar 
\sigma(\indicator{F}x;\indicator{F}y)\\
&\indicator{F}\bar b(x;y) = \bar b(x;\indicator{F}y) = \bar 
b(\indicator{F}x;\indicator{F}y)
\end{split}
\end{equation}
By Proposition \ref{suff-ext-unq-z} we have the $ 
(\F^{\xi^{(k)}}_t)$ adapted (and hence $ 
(\F^{\xi}_t)$ adapted) strong solution denoted by $\{Z^{(k)}_t\}$, 
satisfying a.s.
\[Z^{(k)}_t = c +  \int_0^t\bar{\sigma}(Z^{(k)}_s;\xi^{(k)}).\,dB_s + 
\int_0^t\bar{b}(Z^{(k)}_s;\xi^{(k)})\, ds,\, t \geq 0.\]
Using \eqref{F0sets-in-st-intg} and \eqref{bar-sigma-b-indicator}, we have 
a.s. for all $t \geq 0$
\begin{align*}
\indicator{(\|\xi\|_p\leq k)}Z^{(k)}_t 
&= \indicator{(\|\xi\|_p\leq 
k)}c + \int_0^t\bar{\sigma}(\indicator{(\|\xi\|_p\leq 
k)}Z^{(k)}_s;\xi^{(k)}).\,dB_s\\
&+ 
\int_0^t\bar{b}(\indicator{(\|\xi\|_p\leq 
k)}Z^{(k)}_s;\xi^{(k)})\, ds.
\end{align*}
and
\begin{align*}
\indicator{(\|\xi\|_p\leq k)}Z^{(k+1)}_t 
&= \indicator{(\|\xi\|_p\leq 
k)}c + \int_0^t\bar{\sigma}(\indicator{(\|\xi\|_p\leq 
k)}Z^{(k+1)}_s;\indicator{(\|\xi\|_p\leq 
k)}\xi^{(k+1)}).\,dB_s\\
&+ 
\int_0^t\bar{b}(\indicator{(\|\xi\|_p\leq 
k)}Z^{(k+1)}_s;\indicator{(\|\xi\|_p\leq 
k)}\xi^{(k+1)})\, ds\\
&= \indicator{(\|\xi\|_p\leq 
k)}c + \int_0^t\bar{\sigma}(\indicator{(\|\xi\|_p\leq 
k)}Z^{(k+1)}_s;\xi^{(k)}).\,dB_s\\
&+ 
\int_0^t\bar{b}(\indicator{(\|\xi\|_p\leq 
k)}Z^{(k+1)}_s;\xi^{(k)})\, ds
\end{align*}
Using the uniqueness obtained in Proposition \ref{suff-ext-unq-z} (applied to 
$ 
(\F^{\xi}_t)$ adapted processes), we have a.s.
\begin{equation}\label{Zk-consistency}
\indicator{(\|\xi\|_p\leq k)}Z^{(k+1)}_t 
=\indicator{(\|\xi\|_p\leq k)} Z^{(k)}_t , \, t \geq 0
\end{equation}
with the null set possibly depending on $k$. Let $\tilde\Omega_k$ be the set 
of probability $1$ where the above relation holds. Then 
on $\Omega' := \bigcap_{k=1}^\infty \tilde\Omega_k$, which is a set of 
probability $1$, \eqref{Zk-consistency} holds for all $k$.\\
Note that $(\|\xi\|_p < \infty) = \Omega$ and hence for any $\omega \in 
\Omega$, there 
exists a positive integer $k$ such that $\|\xi(\omega)\|_p \leq k$. Then we can 
write $\Omega' = \bigcup_{k=1}^\infty \left(\Omega' \cap 
(\|\xi\|_p \leq k)\right)$. Now $\Omega'$ is an element of $\F$ with 
probability $1$ and hence 
$(\Omega')^c$ is a null set in $\F$. Since $(\F_t)$ satisfies the usual 
conditions, we have $(\Omega')^c \in \F_0$ and hence $\Omega'  
\in \F_0$.\\
We define a 
process $\{X_t\}$ as follows: for any $t \geq 0$
\[X_t(\omega) := \begin{cases}
Z^{(k)}_t(\omega),\,\text{if}\, \omega \in \Omega' \cap 
(\|\xi\|_p \leq k),\, k = 1,2,\cdots\\
0,\, \text{if}\, \omega \in (\Omega')^c\\
\end{cases}
\]
From equation \eqref{Zk-consistency}, $Z^{(k+1)}_t = Z^{(k)}_t,\, \forall t 
\geq 0$ on 
$\Omega' \cap 
(\|\xi\|_p \leq k)$ and hence $\{X_t\}$ is well-defined. Furthermore $\{X_t\}$ 
is $(\F^{\xi}_t)$ adapted and has continuous paths. We now show 
that $\{X_t\}$ solves 
equation \eqref{spdeinRd-randomcoeff}.
On $\Omega'$ we 
have
\begin{equation}\label{X-Zk-consistency}
\indicator{(\|\xi\|_p\leq k)} X_t = \indicator{(\|\xi\|_p\leq k)} 
Z^{(k)}_t,\, \forall t \geq 0, k=1,2,\cdots
\end{equation}
i.e. above relation holds almost surely. Then for each $k=1,2,\cdots$, a.s. $t 
\geq 
0$
\begin{align*}
\indicator{(\|\xi\|_p\leq k)} X_t &= \indicator{(\|\xi\|_p\leq 
k)} Z^{(k)}_t\\
&= \indicator{(\|\xi\|_p\leq 
k)}c + \int_0^t\bar{\sigma}(\indicator{(\|\xi\|_p\leq 
k)}Z^{(k)}_s;\xi^{(k)}).\,dB_s\\
&+ 
\int_0^t\bar{b}(\indicator{(\|\xi\|_p\leq 
k)}Z^{(k)}_s;\xi^{(k)})\, ds\\
&= \indicator{(\|\xi\|_p\leq 
k)} c+ \int_0^t\bar{\sigma}(\indicator{(\|\xi\|_p\leq 
k)}X_s;\xi^{(k)}).\,dB_s\\
&+ 
\int_0^t\bar{b}(\indicator{(\|\xi\|_p\leq 
k)}X_s;\xi^{(k)})\, ds,\,(\text{using}\, 
\eqref{X-Zk-consistency})\\
&= \indicator{(\|\xi\|_p\leq 
k)} c+ \int_0^t\indicator{(\|\xi\|_p\leq 
k)} \bar{\sigma}(X_s;\xi).\,dB_s\\
&+ 
\int_0^t\indicator{(\|\xi\|_p\leq 
k)} \bar{b}(X_s;\xi)\, ds,\,(\text{using}\, 
\eqref{bar-sigma-b-indicator})\\
&= \indicator{(\|\xi\|_p\leq 
k)} c + \indicator{(\|\xi\|_p\leq 
k)} \int_0^t \bar{\sigma}(X_s;\xi).\,dB_s\\
&+ \indicator{(\|\xi\|_p\leq 
k)}
\int_0^t \bar{b}(X_s;\xi)\, ds,\,(\text{using}\, 
\eqref{F0sets-in-st-intg})
\end{align*}
Let $\bar\Omega_k$ denote the set of probability $1$ where the above relation 
holds. Then $\bar\Omega := \bigcap_{k=1}^\infty \bar\Omega_k$ is also a set of 
probability $1$ and on $\bar\Omega$, for all $k=1,2,\cdots$ and for all $t 
\geq 0$
\[\indicator{(\|\xi\|_p\leq k)} X_t = \indicator{(\|\xi\|_p\leq 
k)}\left( c +  \int_0^t \bar{\sigma}(X_s;\xi).\,dB_s + 
\int_0^t \bar{b}(X_s;\xi)\, ds \right).\]
Then on $\bar\Omega \cap (\|\xi\|_p\leq k)$ we have for all $t \geq 0$
\[ X_t =  c +  \int_0^t \bar{\sigma}(X_s;\xi).\,dB_s + 
\int_0^t \bar{b}(X_s;\xi)\, ds.\]
But $\bar\Omega \cap (\|\xi\|_p\leq k) \uparrow \bar\Omega$ and hence on 
$\bar\Omega$ above relation holds for all $t \geq 0$. So $\{X_t\}$ is a 
solution of \eqref{spdeinRd-randomcoeff}.\\
Taking $\eta_n:=\inf\{t \geq 0: |X_t|\geq 
n\}$ it follows that $\Exp \int_0^{t\wedge\eta_n} 
|X_t|^2\,dt < \infty$ for any $t>0$.\\
To prove the uniqueness, let $\{\tilde X_t\}$ be a continuous $(\F^\xi_t)$ 
adapted strong solution of \eqref{spdeinRd-randomcoeff}. Then a.s. for all $t 
\geq 0$
\begin{align*}
&\indicator{(\|\xi\|_p\leq k)}\tilde X_t\\
&=  \indicator{(\|\xi\|_p\leq k)} \left(c +  
\int_0^t \bar{\sigma}(\tilde X_s;\xi).\,dB_s + 
\int_0^t \bar{b}(\tilde X_s;\xi)\, ds\right)\\
&=  \indicator{(\|\xi\|_p\leq k)}c +  
\int_0^t \bar{\sigma}(\indicator{(\|\xi\|_p\leq k)}\tilde X_s;\xi^{(k)}).\,dB_s 
+ 
\int_0^t \bar{b}(\indicator{(\|\xi\|_p\leq k)}\tilde X_s;\xi^{(k)})\, ds
\end{align*}
From the uniqueness obtained in Proposition \ref{suff-ext-unq-z} and 
equation \eqref{X-Zk-consistency}, we now conclude a.s. for all $t \geq 
0$, $\indicator{(\|\xi\|_p\leq k)}\tilde X_t = \indicator{(\|\xi\|_p\leq k)} 
Z^{(k)}_t = \indicator{(\|\xi\|_p\leq k)} X_t$. Since $(\|\xi\|_p\leq k) 
\uparrow \Omega$, this proves $P(X_t 
= 
\tilde X_t, \, t \geq 0)=1$.
\end{proof}

In Theorem \ref{suff-ext-unq-z2} we can assume locally Lipschitz 
nature of the coefficients $\bar\sigma,\bar b$ instead of those being globally 
Lipschitz. This extension from globally Lipschitz to locally Lipschitz is a 
well-known technique (see \cite[Theorem 18.3 and the discussion in page 
340 about explosion]{MR1464694}, \cite[Chapter IX, 
Exercise 2.10]{MR1725357}, \cite[Theorem 2.3 and 3.1]{MR1011252}). The one point 
compactification of $\R^d$ is denoted by 
$\widehat{\R^d}:=\R^d\cup\{\infty\}$. We state the next result without proof.

\begin{thm}\label{suff-loc-lip}
Suppose the following are satisfied.
\begin{enumerate}
[label=(\roman*)]
\item $\Exp\|\xi\|_p^2 < \infty$.
\item $\zeta=0$.
\item (Locally Lipschitz in x, locally in y) for any fixed $y 
\in \Sc_p(\R^d)$ the functions $x 
\mapsto 
\bar{\sigma}(x;y)$ and $x 
\mapsto 
\bar{b}(x;y)$ are locally Lipschitz functions in $x$ and the Lipschitz 
coefficient is independent of $y$ when $y$ varies over any bounded set $G$ 
in 
$\Sc_p(\R^d)$; i.e. for any bounded set $G$ in 
$\Sc_p(\R^d)$ and any positive integer $n$ there exists a constant $C(G,n) > 0$ 
such that for all $x_1,x_2\in B(0,n), y \in G$
\[|\bar{\sigma}(x_1;y) - \bar{\sigma}(x_2;y)|+|\bar{b}(x_1;y) - \bar{b}(x_2;y)| 
\leq C(G,n) |x_1-x_2|,\]
where $B(0,n)=\{x \in \R^d: |x| \leq n\}$.
\end{enumerate}
Then there exists an $(\F_t^{\xi})$ 
stopping 
time $\eta$ and an $(\F_t^{\xi})$ adapted $\widehat{\R^d}$ valued process 
$\{X_t\}$ such that
\begin{enumerate}[label=(\alph*)]
\item $\{X_t\}$ solves \eqref{spdeinRd-randomcoeff} upto $\eta$ i.e. a.s.
\[X_t = \int_0^t\bar{\sigma}(X_s;\xi).\,dB_s + \int_0^t\bar{b}(X_s;\xi)\, ds,\, 
0 \leq t < \eta\]
and $X_t = \infty$ for $t 
\geq \eta$.
\item $\{X_t\}$ has continuous paths on the interval $[0,\eta)$.
\item $\eta = \lim_m \theta_m$ where $\{\theta_m\}$ are $(\F_t^{\xi})$ stopping 
times defined by $\theta_m := \inf\{t \geq 0: |X_t| \geq m\}$.
\end{enumerate}
This is also pathwise unique 
in this sense: if $(\{X_t'\},\eta')$ is another solution satisfying $(a), (b), 
(c)$, then 
$P(X_t=X_t', 0\leq t < \eta\wedge\eta')=1$.
\end{thm}

In Proposition \ref{loc-lip-example} we show that stronger assumption on $\xi$ 
implies a 'local Lipschitz' 
condition. We use this result to obtain Theorem \ref{stronger-xi-loc-lip}, which 
is a version of Theorem \ref{suff-loc-lip}. Fix $p > 
d+\frac{1}{2}$ and $y \in \Sc_p(\R^d)$. Note that $\delta_x \in 
\Sc_{-p}(\R^d), \, \forall x \in \R^d$ (see \cite[Theorem 4.1]{MR2373102}). 
Hence $x 
\mapsto \inpr{\delta_x}{y}: \R^d \to \R$ is well-defined. Abusing notation, we 
denote this function by $y$. Next result is about the continuity and 
differentiability of the function $y$.
\begin{propn}\label{differentiability-deltax-y}
Let $p, y$ be as above. Then the first order partial derivatives of function 
$y$ 
exist and the distribution 
$y$ is given by the differentiable function $y$. Furthermore, the first order 
distributional derivatives of $y$ are given by the first order partial 
derivatives of $y$, which are continuous functions.
\end{propn}
\begin{proof}
The tempered distribution $y$ can be expressed as $y 
\overset{\Sc_p(\R^d)}{=} 
\sum_{k=0}^\infty\sum_{|n|=k}y_n h_n$ for some $y_n \in \R$. Note that
\begin{enumerate}
\item The Hermite functions $h_n$ are uniformly bounded (see \cite{MR0372517}).
\item From $\|y\|_p^2 = 
\sum_{k=0}^\infty\sum_{|n|=k}(2k+d)^{2p} y_n^2$, we get $|y_n| \leq \|y\|_p 
(2|n|+d)^{-p}$.
\item By \cite[Chapter II, Section 5]{MR0228020}, the cardinality $\#\{n\in 
\mathbb{Z}^d_+ : |n| = \binom{k+d-1}{d-1}$. Hence $\#\{n\in 
\mathbb{Z}^d_+ : |n| = 
k\} \leq C'.(2k+d)^{d-1}$ for some $C' > 0$.
\end{enumerate}
Using these estimates we can show the convergences of 
$\sum_{k=0}^\infty\sum_{|n|=k} y_n h_n(x)$ and $\sum_{k=0}^\infty\sum_{|n|=k} 
y_n \partial_i h_n(x)$ are uniform in $x$. The required continuity and 
differentiability follows from properties of uniform convergence.
\end{proof}

\begin{propn}\label{loc-lip-example}
Let $p > d+\frac{1}{2}$ and $\sigma \in \Sc_{-p}(\R^d)$. Then for any bounded 
set $G$ in 
$\Sc_{p+\frac{1}{2}}(\R^d)$ and any positive integer $n$ there exists a 
constant $C(G,n) > 0$ 
such that for all $x_1,x_2\in B(0,n), y \in G$
\[|\bar{\sigma}(x_1;y) - \bar{\sigma}(x_2;y)|
\leq C(G,n) |x_1-x_2|,\]
where $B(0,n)=\{x \in \R^d: |x| \leq n\}$. 
\end{propn}
\begin{proof}
Let $x^1 = (x^1_1,\cdots,x^1_d), x^2 = (x^2_1,\cdots,x^2_d) \in B(0,n)$. Then 
for any $y \in \Sc_p(\R^d)$,
\begin{equation}\label{intermediate-esmt-loc-lip}
|\bar\sigma(x_1;y)-\bar\sigma(x_2;y)| \leq \|\sigma\|_{-p} \|\tau_{x_1} y 
- \tau_{x_2} y\|_p.
\end{equation}
The target of the subsequent computations is to obtain an estimate of 
$\|\tau_{x_1} y 
- \tau_{x_2} y\|_p$. Note that the first order 
distributional derivatives of $y$ are given by the first order partial 
derivatives of $y$, which are continuous functions (see Proposition 
\ref{differentiability-deltax-y}). For any $1 \leq i \leq d$ and 
$t = \lambda_i 
x^1_i+(1-\lambda_i)x^2_i$ with $\lambda_i \in [0,1]$ we can show
\[|(x^1_1,\cdots,x^1_{i-1},t,x^2_{i+1},\cdots,x^2_d)| \leq 2n.\]
Let $y\in \Sc_{p+\frac{1}{2}}(\R^d)$. Then by Lemma \ref{tau-x-estmte}, there 
exist constants 
$C_n > 0, \tilde C_n > 0$ independent of $i$ such that
\begin{equation}\label{mv-tau-esmt}
\begin{split}
\|\tau_{(x^1_1,\cdots,x^1_{i-1},t,x^2_{i+1},\cdots,x^2_d)} 
\partial_i y\|_p &= 
\|\partial_i \tau_{(x^1_1,\cdots,x^1_{i-1},t,x^2_{i+1},\cdots,x^2_d)} y\|_p\\
&\leq C_n\,\|\tau_{(x^1_1,\cdots,x^1_{i-1},t,x^2_{i+1},\cdots,x^2_d)} 
y\|_{p+\frac{1}{2}} \leq \tilde C_n\, \|y\|_{p+\frac{1}{2}}
\end{split}
\end{equation}
The following is an equality of continuous functions.
\begin{align*}
&\tau_{(x^1_1,\cdots,x^1_{i-1},x^1_i,x^2_{i+1},\cdots,x^2_d)} 
y(\cdot) - \tau_{(x^1_1,\cdots,x^1_{i-1},x^2_i,x^2_{i+1},\cdots,x^2_d)} 
y(\cdot)\\
&=\int_{x^2_i}^{x^1_i}\tau_{(x^1_1,\cdots,x^1_{i-1},t,x^2_{i+1},\cdots,x^2_d)}
\partial_i y(\cdot)\, dt
\end{align*}
In view of \eqref{mv-tau-esmt}, we have the equality 
of distributions in $\Sc_p(\R^d)$
\begin{equation*}
\begin{split}
&\tau_{(x^1_1,\cdots,x^1_{i-1},x^1_i,x^2_{i+1},\cdots,x^2_d)} 
y - \tau_{(x^1_1,\cdots,x^1_{i-1},x^2_i,x^2_{i+1},\cdots,x^2_d)} 
y\\
&= 
\int_{x^2_i}^{x^1_i} \tau_{(x^1_1,\cdots,x^1_{i-1},t,x^2_{i+1},\cdots,x^2_d)}
\partial_i y\, dt
\end{split}
\end{equation*}
and
\begin{align*}
&\|\tau_{(x^1_1,\cdots,x^1_{i-1},x^1_i,x^2_{i+1},\cdots,x^2_d)} 
y - \tau_{(x^1_1,\cdots,x^1_{i-1},x^2_i,x^2_{i+1},\cdots,x^2_d)} 
y\|_p\\
&\leq 
\left| \int_{x^2_i}^{x^1_i}
\|\tau_{(x^1_1,\cdots,x^1_{i-1},t,x^2_{i+1},\cdots,x^2_d)}
\partial_i y\|_p\, dt\right| \leq \tilde C_n\, \|y\|_{p+\frac{1}{2}} \, |x^1_i 
- x^2_i|.
\end{align*}
Now
\begin{align*}
\tau_{x_1} y - \tau_{x_2} y &= 
\tau_{(x^1_1,\cdots,x^1_{d-1}, x^1_d)} 
y - \tau_{(x^1_1,\cdots,x^1_{d-1}, x^2_d)} 
y\\
& + \tau_{(x^1_1,\cdots,x^1_{d-1}, x^2_d)} 
y - \tau_{(x^1_1,\cdots,x^1_{d-2}, x^2_{d-1}, x^2_d)} 
y\\
& + \cdots\\
& + \tau_{(x^1_1,x^2_2,\cdots, x^2_d)} 
y - \tau_{(x^2_1,\cdots, x^2_d)} 
y
\end{align*}
and hence $\|\tau_{x_1} y - \tau_{x_2} y\|_p \leq \tilde C_n\, 
\|y\|_{p+\frac{1}{2}} \, \sum_{i=1}^d |x^1_i - x^2_i| \leq d\tilde C_n\, 
\|y\|_{p+\frac{1}{2}} \, |x_1 - x_2|$. Using this estimate 
in \eqref{intermediate-esmt-loc-lip}, we have 
$|\bar\sigma(x_1;y)-\bar\sigma(x_2;y)| \leq  d\tilde C_n\, \|\sigma\|_{-p}
\|y\|_{p+\frac{1}{2}} \, |x_1 - x_2|$. In particular, if $G$ is a bounded set 
in 
$\Sc_{p+\frac{1}{2}}(\R^d)$, then for any $y \in G$
\[\|\tau_{x_1} y - \tau_{x_2} y\|_p \leq d\tilde C_n\, \|\sigma\|_{-p}
\sup_{y \in G}(\|y\|_{p+\frac{1}{2}}) \, |x_1 - x_2|,\]
i.e. the function $x \mapsto 
\bar\sigma(x;y)$ is locally Lipschitz in $x$ for any $y \in G$ and that the 
Lipschitz constant can be taken uniformly in $y \in G$.
\end{proof}

Using Proposition \ref{loc-lip-example}, we get the following version of 
Theorem \ref{suff-loc-lip}.
\begin{thm}\label{stronger-xi-loc-lip}
Let $p > d+\frac{1}{2}$. Suppose the following are satisfied.
\begin{enumerate}
\item $\sigma,b \in 
\Sc_{-p}(\R^d)$.
\item $\xi$ is $\Sc_{p+\frac{1}{2}}(\R^d)$ valued and $\Exp 
\|\xi\|^2_{p+\frac{1}{2}} < \infty$.
\item $\zeta = 0$.
\end{enumerate}
Then there exists an $(\F_t^{\xi})$ 
stopping 
time $\eta$ and an $(\F_t^{\xi})$ adapted $\widehat{\R^d}$ valued process 
$\{X_t\}$ such that
\begin{enumerate}[label=(\alph*)]
\item $\{X_t\}$ solves \eqref{spdeinRd-randomcoeff} upto $\eta$ i.e. a.s.
\[X_t = \int_0^t\bar{\sigma}(X_s;\xi).\,dB_s + \int_0^t\bar{b}(X_s;\xi)\, ds,\, 
0 \leq t < \eta\]
and $X_t = \infty$ for $t 
\geq \eta$.
\item $\{X_t\}$ has continuous paths on the interval $[0,\eta)$.
\item $\eta = \lim_m \theta_m$ where $\{\theta_m\}$ are $(\F_t^{\xi})$ stopping 
times defined by $\theta_m := \inf\{t \geq 0: |X_t| \geq m\}$.
\end{enumerate}
This is also pathwise unique 
in this sense: if $(X_t',\eta')$ is another solution satisfying $(a), (b), 
(c)$, then 
$P(X_t=X_t', 0\leq t < \eta\wedge\eta')=1$.
\end{thm}

We are ready to prove the main result of this section. We make two definitions 
extending \cite[Definition 3.1 and Definition 3.3]{MR3063763}. Note that $\xi$ 
is assumed to be 
independent of the Brownian motion $\{B_t\}$ and $\hat\Sc_p(\R^d) = 
\Sc_p(\R^d) \cup \{\delta\}$, where $\delta$ is an isolated point (see Section 
1).
\begin{defn}
\begin{enumerate}[label=(\Alph*)]
\item We 
say $\{Y_t\}$ is an $\Sc_p(\R^d)$ valued strong 
solution of equation 
\eqref{spdein-sprime-init}, if  $\{Y_t\}$ is an $\Sc_p(\R^d)$ valued 
$(\F^\xi_t)$ 
adapted continuous process such that a.s. the following equality holds in 
$\Sc_{p-1}(\R^d)$,
\[Y_t = \xi + \int_0^t A(Y_s).\,dB_s + \int_0^t L(Y_s)\, ds;\, t \geq 0.\]
\item By an $\hat\Sc_p(\R^d)$ valued strong local solution of equation 
\eqref{spdein-sprime-init}, we 
mean a pair $(\{Y_t\},\eta)$ where $\eta$ is an $(\F^\xi_t)$ stopping time and 
$\{Y_t\}$ an $\hat\Sc_p(\R^d)$ valued $(\F^\xi_t)$ adapted continuous process 
such that
\begin{enumerate}[label=(\arabic*)]
\item for all $\omega \in \Omega$, the map $Y_{\cdot}(\omega):[0,\eta(\omega)) 
\to \Sc_p(\R^d)$ is continuous and $Y_t(\omega) = \delta, \, t \geq 
\eta(\omega)$.
\item a.s. the following equality holds in $\Sc_{p-1}(\R^d)$,
\[Y_t = \xi + \int_0^t A(Y_s).\,dB_s + \int_0^t L(Y_s)\, ds;\, 0 \leq t < 
\eta.\]
\end{enumerate}
\end{enumerate}

\end{defn}
\begin{defn}
\begin{enumerate}[label=(\Alph*)]
\item We say strong solutions to equation \eqref{spdein-sprime-init} are 
pathwise 
unique, 
if given any 
two $\Sc_p(\R^d)$ valued strong solutions $\{Y_t^1\}$ and 
$\{Y_t^2\}$, we have $P(Y_t^1 = Y_t^2,\, t \geq 0)=1$.
\item We say strong local solutions to equation \eqref{spdein-sprime-init} are 
pathwise unique, 
if given any 
two $\hat\Sc_p(\R^d)$ valued strong solutions $(\{Y_t^1\},\eta^1)$ and 
$(\{Y_t^2\},\eta^2)$, we have $P(Y_t^1 = Y_t^2,\, 0 \leq t < 
\eta^1\wedge\eta^2)=1$.
\end{enumerate}
\end{defn}
Now we prove the existence and uniqueness of solutions to 
\eqref{spdein-sprime-init}.
\begin{thm}\label{exstunq-spdeins'-init}
Suppose the following conditions are satisfied.
\begin{enumerate}
[label=(\roman*)]
\item $\Exp \|\xi\|_p^2 < \infty$.
\item (Globally Lipschitz in x, locally in y) For any fixed $y \in 
\Sc_p(\R^d)$, the functions $x 
\mapsto 
\bar{\sigma}(x;y)$ and $x 
\mapsto 
\bar{b}(x;y)$ are globally Lipschitz functions in $x$ and the Lipschitz 
coefficient is independent of $y$ when $y$ varies over any bounded set $G$ in 
$\Sc_p(\R^d)$; i.e. for any bounded set $G$ in 
$\Sc_p(\R^d)$ there exists a constant $C(G) > 0$ such 
that for all $x_1,x_2\in \R^d, y \in G$
\[|\bar{\sigma}(x_1;y) - \bar{\sigma}(x_2;y)|+|\bar{b}(x_1;y) - \bar{b}(x_2;y)| 
\leq C(G) |x_1-x_2|.\]
\end{enumerate}
Then 
equation \eqref{spdein-sprime-init} has an $(\F^{\xi}_t)$ adapted continuous 
strong solution. The solutions are pathwise unique.
\end{thm}
First we need a characterization of the solution of equation 
\eqref{spdein-sprime-init}. This is an extension of \cite[Lemma 
3.6]{MR3063763} to random initial condition $\xi$. The arguments are similar 
and we state the result without proof.
\begin{lem}\label{soln-characterization}
Let $\xi,\bar\sigma,\bar b$ be as in Theorem \ref{exstunq-spdeins'-init}. Let 
$\{Y_t\}$ be an $(\F^{\xi}_t)$ adapted $\Sc_p(\R^d)$ valued strong solution of 
\eqref{spdein-sprime-init}. 
Define a process 
$\{Z_t\}$ as follows:
\[Z_t := \int_0^t \inpr{\sigma}{Y_s}\, dB_s + \int_0^t \inpr{b}{Y_s}\, ds,\, 
t \geq 0.\]
Then a.s. $Y_t = \tau_{Z_t}\xi$ for $t \geq 0$ and consequently $Z$ solves 
equation \eqref{spdeinRd-randomcoeff} with $Z_0 = 0$.
\end{lem}

\begin{proof}[{Proof of Theorem \ref{exstunq-spdeins'-init}}]
The 
proof is similar to that of \cite[Theorem 3.4]{MR3063763}. By Theorem 
\ref{suff-ext-unq-z2}, we have a solution $\{Z_t\}$ of 
\eqref{spdeinRd-randomcoeff} with initial condition $Z_0=0$. Then 
using the It\={o} formula in Theorem \ref{Ito-random}, we observe that the 
process $\{\tau_{Z_t}\xi\}$ is a solution.\\
To prove the uniqueness, let $\{Y^1_t\}, \{Y^2_t\}$ be two solutions. Then 
define $\{Z^1_t\}$ and $\{Z^2_t\}$ corresponding to $\{Y^1_t\}, \{Y^2_t\}$ as 
in 
Lemma \ref{soln-characterization}. The uniqueness 
part in Theorem \ref{suff-ext-unq-z2} implies a.s. $Z^1_t =Z^2_t, \forall t 
\geq 0$ and hence a.s. $Y^1_t =Y^2_t, \forall t 
\geq 0$.
\end{proof}
Since $Y_t = \tau_{Z_t}\xi$ solves \eqref{spdein-sprime-init} (notations as in 
Theorem \ref{exstunq-spdeins'-init}), we have $\Exp \|Y_0\|_p^2 = \Exp 
\|\xi\|_p^2 < \infty$. Now we prove estimates on $Y_t$ using two different 
techniques.
\begin{propn}\label{Yt-cont-xi}
There exists a localizing sequence $\{\eta_n\}$ such that
\[\Exp \sup_{t \geq 0}\|Y^{\eta_n}_t\|_p^2 \leq C_n.\Exp 
\|Y_0\|_p^2,\]
where the constant $C_n$ depends only on $n$.
\end{propn}
\begin{proof}
Consider the process $\{Z_t\}$ as in Lemma \ref{soln-characterization}. Define 
a 
localizing sequence 
$\{\eta_n\}$ as follows: $\eta_n:=\inf\{t\geq 0: |Z_t|\geq n\}, \, n \geq 
1$. Now using Lemma \ref{tau-x-bnd} there exists a polynomial $Q$ of degree 
$2([|p|]+1)$ such that
\[\|Y_t^{\eta_n}\|_p \leq \|\xi\|_p.Q(|Z_t^{\eta_n}|) \leq 
\|\xi\|_p \sup_{\{x:|x|\leq n\}}Q(|x|).\]
Hence $\sup_{t\geq 0}\|Y_t^{\eta_n}\|_p^2 \leq C_n\, \|\xi\|_p^2$ with $C_n = 
(\sup_{\{x:|x|\leq n\}}Q(|x|))^2$. This implies the required estimate.
\end{proof}

Following \cite[Lemma 1]{MR2479730}, we get the next estimate.
\begin{propn}\label{Yt-cont-xi2}
There exists a localizing sequence $\{\eta_n\}$ such 
that for any positive real number $T$,
\[\Exp \sup_{t \leq T}\|Y^{\eta_n}_t\|_{p-1}^2 \leq C.\Exp 
\|Y_0\|_{p-1}^2,\]
where the constant $C$ depends only on $n$ and $T$.
\end{propn}

\begin{proof}
Define three localizing sequences. For any positive integer $n$, consider
\[\bar\eta_n:=\inf\{t\geq 0:\|Y_t-Y_0\|_{p} \geq n\},\]
and
\[\eta_n':=\inf\{t\geq 0:
|\!\inpr{\sigma}{Y_t}\!|\geq n,\, \text{or}\, 
|\!\inpr{b}{Y_t}\!|\geq n\},\]
and  $\eta_n:=\bar\eta_n\wedge\eta_n'$. Now using It\={o} formula for 
$\|\cdot\|^2_{p-1}$ we obtain a.s. $t \geq 0$
\begin{equation}\label{Ito-sq-norm}
\begin{split}
\|Y_t^{\eta_n}\|^2_{p-1} &= 
\|Y_0\|^2_{p-1} + \int_0^{t\wedge\eta_n} 2 
\sum_{i=1}^d\inpr[p-1]{Y_s^{\eta_n}}{A_i Y_s^{\eta_n}}\, 
dB^{(i)}_s\\
&+\int_0^{t\wedge\eta_n} 
\big[2\inpr[p-1]{Y_s^{\eta_n}}{LY_s^{\eta_n}} + 
\sum_{i=1}^d \|A_iY_s^{\eta_n}\|_{p-1}^2\big]\,ds
\end{split}
\end{equation}
where $\{\int_0^{t\wedge\eta_n} 2 
\sum_{i=1}^d\inpr[p-1]{Y_s^{\eta_n}}{A_i Y_s^{\eta_n}}\, 
dB^{(i)}_s\}$ is a continuous martingale and $B^{(i)}_t$ denotes the $i$-th 
component of $B_t$. Then using the Monotonicity inequality 
(\cite[Theorem 2.1 and Remark 3.1]{MR2590157}) and 
taking 
expectation in \eqref{Ito-sq-norm}
\[\Exp\|Y_t^{\eta_n}\|^2_{p-1} \leq \Exp\|Y_0\|^2_{p-1} + 
\gamma \int_0^t \Exp\|Y_s^{\eta_n}\|^2_{p-1}\, ds\]
where the constant $\gamma$ depends only on $\eta_n$. Then Gronwall's 
inequality implies
\begin{equation}\label{sq-norm-expectation-estmt}
\Exp\|Y_t^{\eta_n}\|^2_{p-1} \leq e^{\gamma t}.
\Exp\|Y_0\|^2_{p-1}, \, t \geq 0.
\end{equation}
Let $\{M_t\}$ and $\{V_t\}$ respectively denote the martingale term and the 
finite variation term on the right hand side of \eqref{Ito-sq-norm}. Then using 
the Monotonicity inequality and \eqref{sq-norm-expectation-estmt}, we get
\begin{equation}\label{bnd-Vt}
\Exp \sup_{t \leq T} V_t \leq \gamma\, \Exp \sup_{t \leq T}\int_0^t
\|Y_s^{\eta_n}\|^2_{p-1} \,ds= \gamma\, 
\int_0^T 
\Exp\|Y_s^{\eta_n}\|^2_{p-1} \,ds\leq \widetilde C \,
\Exp\|Y_0\|^2_{p-1}
\end{equation}
for some constant $\widetilde C$ depending only on $\eta_n$ and $T$.\\
By \cite[Theorem 2.5]{MR3331916}, for each $1 \leq i \leq d$, there exists 
a bounded operator 
$\mathbb{T}_i:\Sc_{p-1}(\R^d)\to\Sc_{p-1}(\R^d)$ such that
\begin{equation}\label{estmt-AYt}
\begin{split}
|2\inpr[p-1]{Y_t^{\eta_n}}{A_i 
Y_t^{\eta_n}}| &= 
\left|-2\sum_{j=1}^d 
\inpr{\sigma_{ji}}{Y_t^{\eta_n}}\inpr[p-1]{Y_t^{\eta_n}}{\partial_j 
Y_t^{\eta_n}}\right|\\
&= \left|\sum_{j=1}^d 
\inpr{\sigma_{ji}}{Y_t^{\eta_n}}\inpr[p-1]{Y_t^{\eta_n}}{\mathbb{T}_j 
Y_t^{\eta_n}}\right|\\
&\leq n \sum_{j=1}^d 
|\inpr[p-1]{Y_t^{\eta_n}}{\mathbb{T}_j 
Y_t^{\eta_n}}| \leq \beta 
\|Y_t^{\eta_n}\|_{p-1}^2
\end{split}
\end{equation}
where $\beta = nd\, \max 
\{\|\mathbb{T}_j\|_{\Sc_{p-1}(\R^d)\to\Sc_{p-1}(\R^d)}\mid 1 \leq j \leq 
d\}$. Using  
\eqref{estmt-AYt}, we can obtain (as in \cite[Lemma 1]{MR2479730}) 
\begin{equation}\label{bnd-Mt}
\Exp \sup_{t \leq T} |M_t| \leq \frac{1}{2}\Exp\sup_{t\leq 
T}\|Y_t^{\eta_n}\|_{p-1}^2 + C. \Exp \|Y_0\|_{p-1}^2.
\end{equation}
for some $C > 0$ depending only on $\eta_n$ and $T$. Using 
\eqref{Ito-sq-norm}, \eqref{bnd-Vt} and \eqref{bnd-Mt} we get the desired 
estimate.
\end{proof}

The counterpart of Theorem \ref{exstunq-spdeins'-init} involving locally 
Lipschitz coefficients is as follows. This result is an extension of 
\cite[Theorem 3.4]{MR3063763}. \begin{thm}\label{exstunq-spdeins'-init-loclip}
Suppose the following conditions are satisfied.
\begin{enumerate}
[label=(\roman*)]
\item $\Exp \|\xi\|_p^2 < \infty$.
\item (Locally Lipschitz in x, locally in y) for any fixed $y 
\in \Sc_p(\R^d)$ the functions $x 
\mapsto 
\bar{\sigma}(x;y)$ and $x 
\mapsto 
\bar{b}(x;y)$ are locally Lipschitz functions in $x$ and the Lipschitz 
coefficient is independent of $y$ when $y$ varies over any bounded set $G$ 
in 
$\Sc_p(\R^d)$; i.e. for any bounded set $G$ in 
$\Sc_p(\R^d)$ and any positive integer $n$ there exists a constant $C(G,n) > 0$ 
such that for all $x_1,x_2\in B(0,n), y \in G$
\[|\bar{\sigma}(x_1;y) - \bar{\sigma}(x_2;y)|+|\bar{b}(x_1;y) - \bar{b}(x_2;y)| 
\leq C(G,n) |x_1-x_2|,\]
where $B(0,n)=\{x \in \R^d: |x| \leq n\}$.
\end{enumerate}
Then an $(\F^{\xi}_t)$ adapted continuous strong local 
solution of \eqref{spdein-sprime-init} exists. The solutions are also 
pathwise unique.
\end{thm}

\section{Stationary Solutions}
In this 
section, we 
investigate existence of stationary solutions of stochastic partial differential equation \eqref{spdein-sprime-init}. First we show that finite dimensional stationary processes can be lifted to infinite dimensional stationary processes via the translation operators $\tau_x$.
\begin{propn}\label{lifting-stationary}
Let $\{Z_t\}$ be an $\R^d$ valued stationary process. Let $\xi$ be an $\Sc_p(\R^d)$ valued random variable (for some $p \in \R$), which is independent of $\{Z_t\}$. Then the process $\{Y_t\}$ defined by $Y_t := \tau_{Z_t}\xi$ is also stationary.
\end{propn}
\begin{proof}
Let $\overset{\mathcal{L}}{=}$ denote equality in law. Since $\{Z_t\}$ is stationary, for time points 
$s,t_1,t_2,\cdots,t_n \geq 0$ we have
\[(Z_{t_1},Z_{t_2},\cdots,Z_{t_n}) 
\overset{\mathcal{L}}{=}(Z_{s+t_1},Z_{s+t_2},\cdots,Z_{s+t_n}).\]
Let $\psi \in \C$. Then $x \mapsto 
\tau_x\psi$ is continuous and and hence measurable (see \cite[proof of 
Proposition 3.1]{MR2373102}). 
Using this fact and the stationarity of $\{Z_t\}$, for Borel sets 
$G_1,\cdots,G_n$ in $\Sc_p(\R^d)$, we have
\begin{equation}\label{stationarity-Z-psi}
\begin{split}
&P((\tau_{Z_{t_1}}\psi,\tau_{Z_{t_2}}\psi,\cdots,\tau_{Z_{t_n}}\psi) \in 
G_1\times G_2\times\cdots\times G_n)\\
&= P((\tau_{Z_{s+t_1}}\psi,\tau_{Z_{s+t_2}}\psi,\cdots,\tau_{Z_{s+t_n}}\psi) 
\in 
G_1\times G_2\times\cdots\times G_n)
\end{split}
\end{equation}
Let $\mu_{\xi}$ denote the law of $\xi$ on $\Sc_p(\R^d)$. Then using conditional 
probability and the independence of $\{Z_t\}$ and $\xi$, we have
\begin{align*}
&P((\tau_{Z_{t_1}}\xi,\tau_{Z_{t_2}}\xi,\cdots,\tau_{Z_{t_n}}\xi) \in 
G_1\times G_2\times\cdots\times G_n)\\
&=\int_{\Sc_p}P((\tau_{Z_{t_1}}\xi,\tau_{Z_{t_2}}\xi,\cdots,\tau_{Z_{t_n}}
\xi) \in 
G_1\times G_2\times\cdots\times G_n | \xi = \psi)\,\mu_{\xi}(d\psi)\\
&=\int_{\C}P((\tau_{Z_{t_1}}\psi,\tau_{Z_{t_2}}\psi,\cdots,\tau_{Z_{t_n}}
\psi) \in 
G_1\times G_2\times\cdots\times G_n)\,\mu_{\xi}(d\psi)
\end{align*}
Similarly,
\begin{align*}
&P((\tau_{Z_{s+t_1}}\xi,\tau_{Z_{s+t_2}}\xi,\cdots,\tau_{Z_{s+t_n}}\xi) \in 
G_1\times G_2\times\cdots\times G_n)\\
&=\int_{\C}P((\tau_{Z_{s+t_1}}\psi,\tau_{Z_{s+t_2}}\psi,\cdots,\tau_{Z_{s+t_n}}
\psi) \in 
G_1\times G_2\times\cdots\times G_n)\,\mu_{\xi}(d\psi).
\end{align*}
Using \eqref{stationarity-Z-psi} we have
\begin{equation}
\begin{split}
&P((\tau_{Z_{t_1}}\xi,\tau_{Z_{t_2}}\xi,\cdots,\tau_{Z_{t_n}}\xi) \in 
G_1\times G_2\times\cdots\times G_n)\\
&= P((\tau_{Z_{s+t_1}}\xi,\tau_{Z_{s+t_2}}\xi,\cdots,\tau_{Z_{s+t_n}}\xi) 
\in 
G_1\times G_2\times\cdots\times G_n)
\end{split}
\end{equation}
i.e. $\{Y_t\}$ is stationary.
\end{proof}
As a consequence of the previous result, stationary solutions of finite dimensional stochastic differential equations can be lifted to stationary solutions of corresponding infinite dimensional stochastic partial differential equations.
\begin{thm}\label{stationary-basic-existence}
Let  $\xi$ be an $\Sc_p(\R^d)$ valued $\F_0$-measurable random variable and independent of $\{B_t\}$. Let $\{Z_t\}$ be a stationary solution of the stochastic differential equation \eqref{spdeinRd-randomcoeff}. Then the process $\{Y_t\}$ defined by $Y_t:=\tau_{Z_t}\xi$ is a stationary solution of the stochastic partial differential equation \begin{equation}\label{spdeins'-splinit}
dY_t = A(Y_t).\,dB_t + L(Y_t)\, dt; \quad
Y_0 = \tau_{Z_0}\xi.
\end{equation}
\end{thm}
\begin{proof}
Using the It\={o} formula in Theorem 
\ref{Ito-random}, we can show that $Y_t = \tau_{Z_t}\xi$ solves 
\eqref{spdeins'-splinit}. Since $\{Z_t\}$ is stationary, so is $\{Y_t\}$ (see Proposition \ref{lifting-stationary}).
\end{proof}

Theorem \ref{stationary-basic-existence} allows us to construct stationary solutions of \eqref{spdeins'-splinit} from those of \eqref{spdeinRd-randomcoeff}. In practice, however, it might be difficult to obtain stationary solutions of the stochastic differential equations \eqref{spdeinRd-randomcoeff}. These difficulties may arise from the coefficients $\bar\sigma, \bar b$, i.e. from the interplay of $\sigma, b$ and the random variable $\xi$. We present a method (see Theorem \ref{stationary-existence}) of constructing stationary solutions of \eqref{spdeins'-splinit} from those of (possibly) unrelated finite dimensional stochastic differential equations by modifying the random variable $\xi$.

Assume that
\begin{enumerate}
\item $f:\R^d\to\R^{d\times d}, g:\R^d\to\R^d$ are measurable functions such 
that the stochastic differential equation
\begin{equation}\label{fd-sde}
dZ_t = f(Z_t)dB_t + g(Z_t)dt, \, \forall t \geq 0
\end{equation}
has a stationary solution and we 
denote the corresponding 
invariant measure by $\nu$. Let $f=(f_{ij}),g=(g_i), 1 
\leq i,j \leq d$ be the component functions of $f,g$.
\item $\sigma_{ij},b_i$ (for $i,j=1,\cdots,d$) are tempered 
distributions given by continuous functions. \end{enumerate}

\begin{rem}
Typically $f,g$ will be locally Lipschitz functions such that explosions do not 
happen in finite time. This non-explosion is guaranteed by a 	`Liapunov' type
criteria (see 
\cite[7.3.14 Corollary]{MR1267569}).
\end{rem}

Note that there exists a $p>0$ such that $\sigma_{ij}, b_i \in 
\Sc_{-p}(\R^d)$ for all $i,j$. Fix such a $p > 0$. Consider the following 
subset of $\Sc_{p}(\R^d)$, \begin{equation}\label{setC}
\begin{split}
\C = \{\psi \in \Sc_p(\R^d):&\int_{\R^d}\sigma_{ij}(y+x)\psi(y)\,dy = 
f_{ij}(x),\forall x \in \R^d;\\ 
&\int_{\R^d}b_i(y+x)\psi(y)\,dy = g_i(x),\forall x \in \R^d, i,j=1,\cdots,d\}.
\end{split}
\end{equation}
Note that $\C$ is a convex set. The motivation behind above conditions requires 
clarification. Firstly, we want 
to choose a subset $\C$ of $\Sc_{p}(\R^d)$ such that the resultant equation 
\eqref{sdeinRd} is the same for all deterministic initial conditions $\psi$ as 
$\psi$ varies 
over the set $\C$. This allows us to think of $\bar\sigma(x;\psi)$ and $\bar 
b(x;\psi)$ as just $\bar\sigma(x)$ and $\bar b(x)$. Secondly, we want $\bar 
\sigma = f$ and $\bar b = g$ which is a choice that allows us to use the 
invariant measure $\nu$ of \eqref{fd-sde}. The set $\C$ considered above 
provides exactly those conditions, which is pointed out in the next result.

\begin{lem}\label{barsigmab-eq-sigmab}
Let $\psi \in \C$. Then $\bar\sigma(x;\psi) =f(x)$ and 
$\bar b(x;\psi) = g(x)$ for all $x \in \R^d$.
\end{lem}

We present the main result of this section.
\begin{thm}\label{stationary-existence}
Let $\xi$ be a $\C$-valued $\F_0$-measurable random variable with 
$\Exp\|\xi\|_p^2 < \infty$ and independent of $\{B_t\}$. 
Then the process $\{Y_t\}$ defined by $Y_t:=\tau_{Z_t}\xi$, is a stationary process and solves \eqref{spdeins'-splinit}, where $\{Z_t\}$ is the stationary solution of \eqref{fd-sde}.
\end{thm}
\begin{proof}
Using Lemma \ref{barsigmab-eq-sigmab} and the It\={o} formula in Theorem 
\ref{Ito-random}, we can show that $Y_t = \tau_{Z_t}\xi$ solves 
\eqref{spdeins'-splinit}. Since $\{Z_t\}$ is stationary, so is $\{Y_t\}$ (see Proposition \ref{lifting-stationary}).
\end{proof}

The next result will be used in Example \ref{eg-st-soln}.
\begin{lem}\label{x-in-Sc-p}
The tempered distribution $x$ given by the function $x \in \R \mapsto x$ 
belongs to $\Sc_{-p}$ for any $p > 
\frac{3}{4}$. The tempered distribution given by the function $b(x):=x^3, 
x \in \R$ 
belongs to $\Sc_{-p}$ for any $p > 
\frac{7}{4}$.
\end{lem}
\begin{proof}
First we show that the tempered 
distribution $1$ given by the constant function $1$ belongs to 
$\Sc_{-p}$ for any $p > \frac{1}{4}$.\\
The Hermite-Sobolev spaces $\Sc_p(\R;\mathbb{C})$ can be defined 
corresponding to the Schwartz space $\Sc(\R;\mathbb{C})$, where the functions 
are complex valued. 
The Fourier transform 
$\widehat\cdot:\Sc(\R;\mathbb{C})\to\Sc(\R;\mathbb{C})$ defined by
\[\widehat\phi(x):=\frac{1}{\sqrt{2\pi}}\int_{\R^d}e^{-ixy}
\phi(y)\,dy,\quad\forall \phi \in \Sc(\R;\mathbb{C})\]
can be extended to $\widehat\cdot:\Sc(\R;\mathbb{C})\to\Sc(\R;\mathbb{C})$ 
via duality. Since 
$\widehat{h}_n = (-i)^{n}h_n$ (see \cite[Appendix A.5, 
equation (A.27)]{MR562914}), 
this gives an isometry 
$\widehat\cdot:\Sc_p(\R;\mathbb{C})\to\Sc_p(\R;\mathbb{C})$. If $T \in 
\Sc'(\R;\mathbb{C})$ is such that 
$\inpr{T}{\phi} \in \R,\, \forall \phi \in \Sc(\R)$ then 
we have $\|\widehat{T}\|_{\Sc_p(\R;\mathbb{C})} = 
\|T\|_{\Sc_p(\R;\mathbb{C})} 
= \|T\|_{\Sc_p}$. Since $\delta_0 \in \Sc_{-p}$ 
for any $p > \frac{1}{4}$ (see \cite[Theorem 4.1(a)]{MR2373102}), we have $1 
\in \Sc_{-p}$ for any $p > \frac{1}{4}$.\\
The operator $M_x$ was defined in Section 2. Observe that for any $\phi \in 
\Sc$ and $p > \frac{1}{4}$, 
\[|\inpr{x}{\phi}| = |\inpr{1}{M_x\phi}|\leq \|1\|_{-p}\|M_x\phi\|_p \leq 
\|1\|_{-p}\|M_x\|_{\Sc_{p+\frac{1}{2}}\to\Sc_{p}}.\,\|\phi\|_{
p+\frac{1}{2 }}.\]
This implies $x \in \Sc_{-p}$ for $p > \frac{3}{4}$. Proof for $b$ is similar.
\end{proof}

\begin{eg}\label{eg-st-soln}
We present two examples where the stationary solutions of \eqref{fd-sde} can 
be lifted to stationary solutions of \eqref{spdeins'-splinit} via Theorem 
\ref{stationary-existence}.
\begin{enumerate}
\item Take $d=1, f(x) \equiv 1, g(x)=-x , \forall x, \sigma=f,b=g$.  It is 
well-known that \eqref{fd-sde} (the 
Ornstein-Uhlenbeck diffusion) has a stationary solution with the following 
initial 
condition:
\begin{equation}
dZ_t = dB_t - Z_t\, dt;\quad Z_0 \sim N\big(0,\tfrac{1}{2}\big),
\end{equation}
where $N(0,\frac{1}{2})$ denotes the law of a Gaussian random variable with 
mean $0$ and variance $\frac{1}{2}$ (see \cite[Example 6.8]{MR1121940}). Note 
that 
$\sigma \in \Sc_{-p}$ for $p > \frac{1}{4}$ and $b \in \Sc_{-p}$ for $p > 
\frac{3}{4}$ (see Lemma \ref{x-in-Sc-p}). Take $p > \frac{3}{4}$. It is 
easy to 
check that $\C =\{\psi \in \Sc_p: \int_{\R}\psi = 1, 
\int_{\R}t\psi(t)\,dt = 0\}$.
$\C$ is non empty since (centered) Gaussian densities satisfy such 
conditions.
\item We take $\Omega = C([0,\infty),\R)$ and use the setup of \cite[Chapter 
VII, 
Sections 3 and 5]{MR1267569}. Consider $f(x) \equiv 1, g(x)=-x^3 , \forall 
x \in \R, \sigma=f,b=g$. Then 
\eqref{fd-sde} has an invariant measure, say $\nu$, given by
\[\nu(B):=c \int_B \exp\left(-\frac{x^4}{2}\right)\, dx\]
for any Borel set $B$ in $\R$, where $c = 
2^{\frac{3}{4}}\left(\Gamma(\frac{1}{4})\right)^{-1}$ is the normalization 
constant. Take $p > \frac{7}{4}$ so that $\sigma, b \in \Sc_{-p}$ (see Lemma 
\ref{x-in-Sc-p}). Then
\[\C =\{\psi \in \Sc_p: \int_{\R}\psi = 1, 
\int_{\R}t\psi(t)\,dt = 0, \int_{\R}t^2\psi(t)\,dt = 0, \int_{\R}t^3\psi(t)\,dt 
= 0\}.\]
$\C$ is non-empty since $\psi_1, \psi_2 
\in \C$ where
\[\psi_1(t)= 
\exp(-t^2)\left[\frac{3}{2\sqrt{\pi}}-\frac{1}{\sqrt{\pi}}t^2\right], 
\psi_2(t)= 
\exp\left(-\frac{t^2}{2}\right)\left[\frac{3}{2\sqrt{2\pi}}-\frac{1}{2\sqrt{2\pi
}
} t^2\right].\]
\end{enumerate}
\end{eg}

We now prove an estimate of a stationary solution $\{Y_t\}$.
\begin{propn}\label{stationary-estimate}
Let $\xi,\{Z_t\}, \{Y_t\}$ be as in Theorem 
\ref{stationary-existence}. In addition assume that $\xi$ is norm-bounded and 
$Z_0$ 
has moments of orders upto $4([|p|]+1)$ and $f,g$ are Lipschitz continuous. Then
\begin{enumerate}[label=(\alph*)]
\item $\Exp 
\|Y_0\|_p^2 = \Exp \|\tau_{Z_0}\xi\|_p^2 < \infty$.
\item $\Exp \sup_{t \leq 
T}\|Y_t\|_p \leq C\,(\Exp 
\|Y_0\|_p^2)^{\frac{1}{2}}$, where $C$ is a positive constant depending only on 
$f, g$ and $T$.
\end{enumerate}
\end{propn}
\begin{proof}
For any norm-bounded $\C$ valued random 
variable $\xi$, we have $\Exp \|\tau_{Z_0}\xi\|_p^2 \leq R\,\Exp P(|Z_0|)$ 
where $R>0$ 
and $P$ is a polynomial of degree $4([|p|]+1)$ (see Lemma 
\ref{tau-x-estmte}). Then by our assumption, $\Exp 
\|\tau_{Z_0}\xi\|_p^2<\infty$.\\
Observe that $Y_t = \tau_{Z_t}\xi = \tau_{Z_t-Z_0}\tau_{Z_0}\xi = 
\tau_{Z_t-Z_0}Y_0$. Using Lemma \ref{tau-x-bnd} we have
\[\|Y_t\|_p \leq \|Y_0\|_p\, P_k(|Z_t-Z_0|),\]
where $P_k$ is some real polynomial of degree $k=2([|p|]+1)$ with non-negative 
coefficients. We use the following estimate to establish 
the result.
\begin{equation}
\Exp \sup_{t\leq T}\|Y_t\|_p \leq (\Exp \|Y_0\|_p^2)^{\frac{1}{2}}\, 
(\Exp \sup_{t \leq T} 
P_k(|Z_t-Z_0|)^2)^{\frac{1}{2}}
\end{equation}
Now a.s.
$Z_t - Z_0 = \int_0^t f(Z_s)\,dB_s +\int_0^t g(Z_s)\,ds, \, t \geq 0$. Using 
stationarity of $Z$ and the BDG inequalities (\cite[Proposition 
15.7]{MR1464694}), we can show $\Exp \sup_{t \leq T} 
P_k(|Z_t-Z_0|)^2$ is bounded. In this estimate we use the assumption on the 
moments of $Z_0$ and linear growth of $f,g$. This completes the proof.
\end{proof}

\begin{rem}\label{setC-properties}
We make a few observations.
\begin{enumerate}[label=(\arabic*)]
\item If the convex set $\C$ (as in equation
\eqref{setC}) has more than one element, then we can consider probability 
measures on $\C$ which are convex combinations of Dirac measures on $\C$. 
By Theorem \ref{stationary-existence}, we have the existence of infinitely 
many 
stationary 
solutions corresponding to each of these probability measures. To rationalize, 
this may be happening due to $\C$ being not 
translation invariant.
\item The set $\C$ may be non-compact. Consider the special case $d=1, f(x) 
\equiv 1, g(x)=-x, \sigma=f,b=g$ and take $p$ sufficiently large so 
that the tempered distribution given by the function $x \mapsto x^2$ is in 
$\Sc_{-p}$. Then the image of 
$\C$ under this tempered distribution (a continuous linear functional on 
$\Sc_p$) contains $(0,\infty)$, the variances of centered Gaussian densities. 
So $\C$ is unbounded and non-compact.
\end{enumerate}
\end{rem}

\begin{rem}
Existence of invariant measure of finite dimensional diffusions and Markov 
processes has been 
studied by many authors (to cite only a few see \cite{MR0494525, MR671239, 
MR0346904, MR0133871, MR2894052}, \cite[Chapter VII, Section 5]{MR1267569}).
\end{rem}

\noindent\textbf{Acknowledgement:} A part of this work was completed by the author as a research scholar at Indian Statistical Institute, Bangalore. The author would like to thank Professor B.
Rajeev, Indian Statistical Institute, Bangalore for valuable suggestions during
the work.

% \bibliographystyle{amsplain}
% \bibliography{inv-mre-ref.bib}

\end{document}